\documentclass{au}
\usepackage{amssymb,latexsym}

\newcommand{\bga}{\boldsymbol{\alpha}}
\newcommand{\bgd}{\boldsymbol{\delta}}
\newcommand{\restr}{\rceil}
\newcommand{\jj}{\vee}
\newcommand{\mm}{\wedge}
\newcommand{\es}{\varnothing}
\DeclareMathOperator{\length}{len}
\newcommand{\tbf}{\textbf}
\def\cng#1=#2(#3){#1\equiv#2\pmod{#3}}
\newtheorem{lemma}{Lemma}

\begin{document}
\title{A technical lemma for congruences of finite lattices}  
\author{G. Gr\"{a}tzer} 
\email[G. Gr\"atzer]{gratzer@me.com}
\address{Department of Mathematics\\
  University of Manitoba\\
  Winnipeg, MB R3T 2N2\\
  Canada}

\urladdr[G. Gr\"atzer]{http://server.maths.umanitoba.ca/homepages/gratzer/}

\date{Aug. 21, 2013}

\begin{abstract}
The classical Technical Lemma for congruences is not difficult to prove
but it is very efficient in its applications.
We present here a Technical Lemma 
for congruences on \emph{finite lattices}.
This is not difficult to prove either but it has already has proved its usefulness in some applications.
\end{abstract}

\maketitle

Here is the classical Technical Lemma for congruences, see
G. Gr\"atzer and E.\,T. Schmidt~\cite{GS58d} and
F. Maeda \cite{fM58}.

\begin{lemma}\label{L:I.3.8}
A reflexive binary relation $\bga$ 
on a lattice $L$ is a congruence relation
if{}f the following three properties are satisfied 
for any $x, y, z, t \in L$:
\begin{enumerate}
\item[(i)]  $\cng x = y(\bga)$\quad if{}f\quad 
$\cng x \mm y = x \jj y(\bga)$.
\item[(ii)]  Let $x \leq y \leq z$; 
then $\cng x = y(\bga)$ and $\cng y = z (\bga)$
imply that $\cng x = z(\bga)$.
\item[(iii)]  $x \leq y$ and $\cng x = y(\bga)$ 
imply that $\cng x \jj t = y \jj t (\bga)$ 
and $\cng x \mm t = {y \mm t}(\bga)$.
\end{enumerate}
\end{lemma}

I stated and proved this Technical Lemma in all my lattice theory books, 
see for instance, \cite[Lemma I.3.8]{GLT2}, \cite[Lemma 1.1]{CFL}, 
and \cite[Lemma 11]{LTF}. Witness the numerous references to this lemma, for instance, in \cite{LTF}. The lemma is not difficult to prove, but it surely saves a lot of computation wherever we need to prove that a reflexive binary relation is a congruence relation.

In some recent research, G. Cz\'edli
and I, see \cite{gC13}, \cite{gC13a} and \cite{gG13a}, \cite{gG13c}
spent quite an effort
in proving that some equivalence relations on a planar semimodular lattices
with intervals as equivalence classes are congruences. 
The number of cases we had to consider
was dramatically cut by the following result.

\begin{lemma}\label{L:technical}
Let $L$ be a finite lattice. 
Let $\bgd$ be an equivalence relation on $L$
with intervals as equivalence classes.
Then $\bgd$ is a congruence relation if{}f the following condition 
and its dual hold:
\begin{equation}\label{E:cover}
\text{If $x$ is covered by $y,z \in L$ 
and $x \equiv y\,(\textup{mod}\, \bgd)$,
then $z \equiv y \jj z\,(\textup{mod}\,\bgd)$.}\tag{C${}_{\jj}$}
\end{equation}
\end{lemma}

\begin{proof}
First, we prove the join-substitution property:  
if $x \leq y$ and $\cng x=y (\bgd)$, then
\begin{equation}\label{E:Cjoin}
\cng x \jj z = y \jj z(\bgd).
\end{equation}

This is trivial if $y = z$, so we assume that $y \neq z$.
Clearly, we can also assume that $x<y$ and $x<z$. 

Let $U = [x, y\jj z]$.
We induct on $\length U$, the length of $U$. 

Using the fact that the intersection of two convex sublattices 
is either $\es$ or a convex sublattice, 
it follows that, for every interval $V$ of $L$, 
the classes of $\bgd \restr V$ are intervals. 
Hence, we can assume that $x=y\mm z$; 
indeed, otherwise the induction hypothesis 
applies to $V = [y \mm z, y \jj z]$ and $\bgd \restr{V}$,
since $\length V < \length U$, yielding \eqref{E:Cjoin}. 

Note that $\length U \geq 2$. 
If $\length U = 2$, 
then \eqref{E:Cjoin} is stated in $(C_\jj)$. 

So we can also assume that $\length U > 2$. 
Pick the elements $y_1,z_1\in L$ so that $x\prec y_1\leq y$ 
and $x \prec z_1 \leq z$. 
The elements $y_1$ and $z_1$ are distinct, 
since $y_1 = z_1$ would contradict that $x = y \mm z = y_1\mm z_1$. 
Let $w = y_1\jj z_1$. 
Since the $\bgd$-classes are intervals, 
$\cng x={y_1}(\bgd)$, 
therefore, (C$_\jj$) yields that 
\begin{equation}\label{E:two}
\cng z_1=w(\bgd).
\end{equation}
Let $I=[y_1,y\jj z]$ and $J=[z_1,y\jj z]$. 
Then $\length I, \length J < \length U$. 
Hence, the induction hypothesis applies to $I$ and $\bgd\restr I$, 
and we obtain that $\cng{w}={y\jj w}(\bgd)$. 
Combining this with \eqref{E:two}, by the transitivity of $\bgd$, 
we conclude that 
\begin{equation}\label{E:three}
\cng{z_1}={y\jj w}(\bgd).
\end{equation}
Therefore, applying the induction hypothesis to $J$ 
and $\bgd \restr J$, we conclude from \eqref{E:three} that 
\[
   \cng {x\jj z = z \jj z_1}={z \jj (y\jj w) = y\jj z}(\bgd),
\] 
proving \eqref{E:Cjoin}.

Second, we get the meet-substitution property by duality.
\end{proof}

Observe that for a finite semimodular lattice (C$_\jj$) states
that we have to check the join-substitution property 
only in covering-square sublattices.

Note that Lemma \ref{L:technical} holds in any lattice $L$
in which every interval has a finite length.

I hope that others, working with congruences of finite lattices, will also find this new Technical Lemma useful.


\begin{thebibliography}{9}

\bibitem{gC13}
Cz\'edli, G.:
\emph{Patch extensions and trajectory colorings of slim
rectangular lattices.}
Algebra Universalis (in press)

\bibitem{gC13a}
Cz\'edli, G.:
\emph{A note on congruence lattices of slim semimodular lattices.}
Algebra Universalis (in press)
  
\bibitem{GLT2}
Gr\"atzer, G.:
General Lattice Theory, second edition.
New appendices by the author with B.\,A. Davey, R. Freese, B. Ganter, M. Greferath, P. Jipsen, H.\,A. Priestley, H. Rose, E.\,T. Schmidt, S.\,E. Schmidt, F. Wehrung, and R. Wille. 
Birkh\"auser, Basel (1998)

\bibitem{CFL}
Gr\"atzer, G.:
The Congruences of a Finite Lattice, A \emph{Proof-by-Picture} Approach. 
Birkh\"auser, Boston (2006)

\bibitem{LTF}
Gr\"atzer, G.:
Lattice Theory: Foundation. 
Birkh\"auser, Basel (2011)

\bibitem{gG13a}
G. Gr\"atzer, 
\emph{The order of principal congruences of a lattice.}
Algebra Universalis \tbf{70}, 95--105 (2013) 

\bibitem{gG13c}
Gr\"atzer, G.:
\emph{Congruences of fork extensions of lattices.}
Acta Sci. Math. (Szeged) (submitted)

\bibitem{GS58d}
Gr\"atzer, G. and Schmidt, E.\,T.:
\emph{Ideals and congruence relations in lattices.} 
 Acta Math. Acad. Sci. Hungar. \tbf{9}, 137--175 (1958)

\bibitem{LTE}
Gr\"atzer, G. and Wehrung, F. eds.:
Lattice Theory: Empire. Special Topics and Applications.
Birkh\"auser, Basel

\bibitem{fM58}
Maeda, F.: 
\emph{Kontinuierliche Geometrien.} 
Springer-Verlag,Heidelberg (1958)
\end{thebibliography}
\end{document}